\documentclass[a4paper,11pt,reqno]{amsart}
 \usepackage{amssymb,amsmath,amstext,amsthm,amsfonts}
\usepackage{euscript,graphicx,a4wide}
\usepackage[shortlabels]{enumitem} 

\usepackage{hyperref}  

\theoremstyle{plain}
\newtheorem{maintheorem}{Theorem}

\newtheorem{maincorollary}[maintheorem]{Corollary}

\newtheorem{theorem}{Theorem}[section]
\newtheorem{proposition}[theorem]{Proposition}
\newtheorem{corollary}[theorem]{Corollary}
\newtheorem{lemma}[theorem]{Lemma}

\newtheorem{remark}[theorem]{Remark}

\theoremstyle{definition}
\newtheorem{definition}{Definition}

\newcommand{\RR}{{\mathbb R}}

\newcommand{\X}{{\EuScript{X}}}

\newcommand{\cG}{{\mathcal G}}

\newcommand{\cM}{{\mathcal M}}

\newcommand{\cO}{{\mathcal O}}
\newcommand{\cP}{{\mathcal P}}

\newcommand{\V}{{\mathcal V}}

\newcommand{\vfi}{{\varphi}}
\renewcommand{\epsilon}{\varepsilon}

\newcommand{\qand}{\quad\text{and}\quad}

\DeclareMathOperator{\diam}{diam}

\DeclareMathOperator{\sing}{Sing}

\DeclareMathOperator{\sgn}{sgn}
\DeclareMathOperator{\m}{Leb}
\DeclareMathOperator{\close}{Closure}

\theoremstyle{remark}

\title[Statistical Stability of the Contracting Lorenz attractor]{On
  the statistical stability of families of attracting sets and the
  contracting Lorenz attractor}

\thanks{
  The author was partially supported by
  CNPq-Brazil (grant 300985/2019-3).}

\date{\today}

\author[Vitor Araujo]{Vitor Araujo}
\email{vitor.araujo.im.ufba@gmail.com, vitor.d.araujo@ufba.br}
\urladdr{https://sites.google.com/site/vdaraujo99/}
\address{Instituto de Matem\'atica e Estat\'{\i}stica,
  Universidade Federal da Bahia, Av. Ademar de Barros s/n,
  40170-110 Salvador, Brazil.}



\begin{document}

\begin{abstract}
  We present criteria for statistical stability of attracting sets for
  vector fields using dynamical conditions on the corresponding
  generated flows. These conditions are easily verified for all
  singular-hyperbolic attracting sets of $C^2$ vector fields using
  known results, providing robust examples of statistically stable
  singular attracting sets (encompassing in particular the Lorenz and
  geometrical Lorenz attractors). These conditions are shown to hold
  also on the persistent but non-robust family of contracting Lorenz
  flows (also known as Rovella attractors), providing examples of
  statistical stability among members of non-open families of
  dynamical systems. In both instances, our conditions avoid the use
  of detailed information about perturbations of the one-dimensional
  induced dynamics on specially chosen Poincar\'e sections.
\end{abstract}

\maketitle

\tableofcontents

\section{Introduction}
\label{sec:intro}

The statistical viewpoint on Dynamical Systems is one of the
cornerstones of most recent developments in dynamics. Given a flow
$\phi_t$ on a manifold $M$, a central concept is that of
\emph{physical measure}, a $\phi_t$-invariant probability measure $\mu$
whose \emph{ergodic basin}
\[
B(\mu)=\left\{x\in M:
\frac1T\int_{0}^{T}\varphi(\phi_tx)\,dt\to\int\varphi\,
d\mu\mbox{  for all continuous  } \varphi: M\to\RR\right\}
\]
has positive \emph{volume} or \emph{Lebesgue measure}, which
we write $\m$ and take as the measure associated with any
non-vanishing volume form on $M$.

This kind of measure provides asymptotic information on a
set of trajectories that one hopes is large enough to be
observable in real-world models. 

The stability of physical measures under small variations of the map
allows for small errors on the formulation of the transformation law
governing the dynamics not to disturb too much the long term behavior,
as measured by the most basic statistical data provided by asymptotic
time averages of continuous functions along orbits.  In principle when
considering practical systems we cannot avoid external noise, so every
realistic mathematical model should exhibit these stability features
to be able to cope with unavoidable uncertainty about the ``correct''
parameter values, observed initial states and even the specific
mathematical formulation involved.

In this note we explicitly state criteria for \emph{statistical
  stability} of families of continuous dynamical systems (flows
generated by vector fields) exhibiting not necessarily robust features
(that is, the family needs not be open in a smooth topology of vector
fields or flows) given by singular attracting sets, namely
singular-hyperbolic or contracting Lorenz models. These families of
invariant sets, containing regular trajectories accumulating
equilibria are not structurally stable, that is, cannot be seen as
different realizations of the same system under a continuous change of
coordinates; see e.g. \cite{GW79}. However, using physical measures we
can obtain stability in a statistical sense: asymptotical time
averages of continuous observables over most trajectores will vary
continuously with the underlying dynamical system.

We first apply the criteria to obtain a rather geometrical proof of
statistical stability for open families singular-hyperbolic (or
Lorenz-like) attracting sets, encompassing in particular de classical
Lorenz attractor and also the family of geometrical Lorenz attractors;
see \cite{AraPac2010} for a presentation of these systems. Our proof
for these systems takes advantage of already known results. Secondly,
we show that the non-open, but persistent, Rovella family
$\cG:X\to\X^3(\RR^3)$ of perturbation of the contracting Lorenz
attractor \cite{Ro93} also satisfies the criteria, where $X$ is a
metric space and $\cG$ is continuous with respect to the $C^3$
topology among smooth vector filds of $\RR^3$. Thus, the physical
measures on these attractors are statistically stable \emph{within the
  family}, that is, when considering perturbations along the image of
the family $\cG$.  We note that recently Alves-Khan~\cite{AlvesKhan}
showed that contracting Lorenz flows are \emph{statistically unstable
  if we consider all the nearby flows} in the $C^3$ topology, that is,
we replace $X$ by an open subset $U$ of $\X^3(\RR^3)$ containg the
contracting Lorenz attractor.

Another notion of stability is that of \emph{stochastic stability},
dealing with small random perturbations along each trajectory, which
we do not consider here, but was studies for sectional-hyperbolic and
contracting Lorenz attractors by Metzger and Morales
\cite{mtz00,MetzMor15}.

Our criteria do not assume uniqueness of the physical measure
supported on the attracting set: we deal with an at most countable
family of ergodic physical measures, as long as their ergodic basins
contain $\m$-almost all points whose trajectories accumulate on the
attracting set. Moreover, the criteria do not involve the statistical
stability of a one-dimensional quotient map induced by a certain
Poincar\'e return map, defined by a suitable choice of global
cross-section, as in the case of the previous works on statistical
stability of geometric Lorenz attractors of Alves-Soufi
\cite{AlveSoufi} and Bahsoum-Ruziboev \cite{bahsoun_ruziboev}.  Our
criteria are a mix of dynamical (robust expansiveness) and
thermodynamical (physical measures satisfy the Entropy Formula)
properties of the flow restricted to the attracting set and its
perturbations.

We mention that statistical stability and other strong properties of
the one-dimensional quotient maps (contracting Lorenz maps) mentioned
above for the Rovella attractor were obtained by Metzger~\cite{mtz001}
and Alves-Soufi~\cite{AlvesSoufi12}. Decay of correlations and other
statistical properties for the Poincar\'e return map were obtained
more recently by Galatolo-Nisoli-Pacifico~\cite{GNP18}, and a
Thermodynamical Formalism for the contracting Lorenz flow was
developed by Pacifico-Todd~\cite{PT09}.

Our results can be immediately applied to certain known families of
bifurcations giving rise to attractors belonging to these two classes: see
e.g.~\cite{Ro2000,MPS05,MPsM}. The family of systems obtained after the
unfolding of these bifurcation scenarios exhibiting
singular-hyperbolic (Lorenz-like) or contracting Lorenz attractors are
automatically statistically stable.

Similar ideas to the criteria presented here, exploring consequences
of the characterization of invariant measures satisfying the Entropy
Formula~\cite{LY85} were already used to deal with stochastic and
statistical stability of uniformly and non-uniformly expanding maps;
see e.g. \cite{Ki88,ArTah,ArTah2} and \cite{araujo2006}. A natural
notion of stability for maps with several physical measures supported
on a given attracting set was provided in \cite{AA03}.  The same
strategy was applied to obtain statistical stability for
sectional-hyperbolic attracting sets (a higher (co)dimensional
extension of the notion of singular-hyperbolicity) in \cite{Araujo19},
where the focus lies on the technically much harder task of deducing
the properties needed to apply the criteria, due to the high
dimensionality of the objects involved.


\subsection{Statements of the results}
\label{sec:statements-results}

Let $M$ be a compact connected Riemannian manifold with dimension
$\dim M=m$, induced distance $d$ and volume form $\m$. Let $\X^r(M)$,
$r\ge1$, be the set of $C^r$ vector fields on $M$ endowed with the
$C^r$ topology and denote by $\phi_t$ the flow generated by
$G\in\X^r(M)$.

\subsubsection{Preliminary definitions}
\label{sec:prelim-definit}

An \emph{invariant set} $\Lambda$ for the flow $\phi_t$ generated by
the vector field $G\in\X^r(M)$, for some fixed $r\ge2$, is a subset of
$M$ which satisfies $\phi_t(\Lambda)=\Lambda$ for all $t\in\RR$.
Given a compact invariant set $\Lambda$ for $G\in \X^r(M)$, we say
that $\Lambda$ is \emph{isolated} if there exists an open set
$U\supset \Lambda$ such that
$ \Lambda =\bigcap_{t\in\RR}\close{\phi_t(U)}$.  If $U$ can be chosen
so that $\close{\phi_t(U)}\subset U$ for all $t>0$, then we say that
$\Lambda$ is an \emph{attracting set} and $U$ a \emph{trapping region}
(or \emph{isolated neighborhood}) for
$\Lambda=\Lambda_G(U)=\cap_{t>0}\close{\phi_t(U)}$.

We note that every attracting set admits a natural continuation, since
there exists a neighborhood $\V$ of $G$ in $\X^r(M)$ so that
$\close{\phi^Y_t}(U)\subset U$ for all $t>0$ and each $Y\in\V$, where
$(\phi_t^Y)_{t\in\RR}$ is the flow generated by $Y$, and so we may
consider the attracting set $\Lambda_Y(U)$.

Physical measures are related to equilibrium states of a certain
potential function. Let $\psi : M \rightarrow \RR$ be a continuous
function.  Then a $\phi_t$-invariant probability measure $\mu$ is a
\emph{equilibrium state for the potential $\psi$} if
\[
P_{G}(\psi)= h_{\mu}(\phi_1) + \int \psi \, d\mu,
\quad\mbox{where}\quad
P_{G}(\psi)=\sup_{\nu \in \cM}
\left\{ h_{\nu}(\phi_1) + \int \psi \, d\nu \right\},
\]
and $\cM$ is the set of all $\phi_t$-invariant probability measures.
The quantity $P_{G}(\phi)$ is called the \emph{Topological Pressure}
and the identity on the right hand side is a consequence of the
\emph{Variational Principle}; see e.g. \cite{Wa82} for definitions of
entropy $h_\mu(\phi_1)$ and topological pressure $P_{G}(\psi)$.

A sign of chaoticity in an attracting set of a vector field is the
property of \emph{expansiveness}. 
  Denote by $S(\RR)$ the set of surjective increasing continuous
  functions $h:\RR\to\RR$. We say that the flow is \emph{expansive} if
  for every $\epsilon>0$ there is $\delta>0$ such that, for any
  $h\in S(\RR)$
  \begin{align*}
    d(\phi_t(x),\phi_{h(t)}(y))\leq\delta,\quad \forall t\in\RR
    \implies
    \exists t_0\in\RR \text{ such that }
    \phi_{h(t_0)}(y)\in \phi_{[t_0-\epsilon,t_0+\epsilon]}(x).
  \end{align*}
We say that a invariant compact set $\Lambda$ is expansive if the
restriction of $\phi_t$ to $\Lambda$ is an expansive flow.

Robust properties are extremely important in Dynamical Systems
theory. To precisely state the main result, we now define robust
expansiveness. Let $\cG:X\to\X^r(M)$ be a continuous family of vector
fields, where $r\ge2$ is fixed and $X$ is a metric space. We write
$G_s=\cG(s)$ the vector field given by $s\in X$ and denote by
$(\phi_t^{ G_s})_{t\in\RR}$ the corresponding flow in what follows.

We say that the family $\cG$ of vector fields is \emph{robustly
  expansive} on an attracting set
$\Lambda=\cap_{t>0}\close{\phi^{G_s}_t(U})$ for some $s\in X$ if there
exists a neighborhood $N$ of $s$ in $X$ such that for every
$\epsilon>0$ there is $\delta>0$ such that, for any
$x,y\in\Lambda_s=\cap_{t>0}\close{\phi^{G_s}_t(U)}$, $h\in S(\RR)$ and
$s\in V$
  \begin{align*}
    d(\phi^{G_s}_t(x),\phi^{G_s}_{h(t)}(y))\leq\delta,\quad \forall t\in\RR
    \implies
    \exists t_0\in\RR \text{ such that }
    \phi_{h(t_0)}(y)\in \phi^{G_s}_{[t_0-\epsilon,t_0+\epsilon]}(x).
  \end{align*}








\subsubsection{Statistical stability of equilibrium states}
\label{sec:statist-stabil-equil}

We can now precisely state our criteria for statistical stability of
families of attracting sets of vector fields.

\begin{maintheorem}
  \label{mthm:statstabeqstates}
  Let us assume that the family $\cG$ admits a trapping region $U$ so
  that the attracting set
  $\Lambda_s(U)=\cap_{t>0}\close{ \phi_t^{G_s}(U)}$ satisfies, for
  each parameter $s$ in some subset $N\subset X$:
  \begin{enumerate}
  \item there are finitely many ergodic physical measures
    $\mu^s_i, 1\le i\le k_i$ supported in $\Lambda_s$ so that\footnote{We
      write $A+B$ the union of the disjoint subsets $A$ and $B$.}
    $\m\big(U\setminus \sum_{i} B(\mu_i^s) \big)=0$;
  \item there exists a family of potentials $\psi_s:\Lambda_s\to\RR$
    so that $\mu$ is an equilibrium state w.r.t. $\psi_s$, i.e.
    $0=h_{\mu}(\phi_1^{G_s})+\int\psi_s\,d\mu$ if, and only if, $\mu$
    a physical measure;
  \item the function
    $\Psi:W(U):=\{(s,x)\in N\times U: x\in\Lambda_s(U)\}\to\RR$ given
    by $\Psi(s,x)=\psi_s(x)$ is continuous; and
\item the family $\cG$ is robustly expansive.
  \end{enumerate}
  Then, for each converging sequence $s_n\in N$ to $s\in N$ and every
  choice $\mu^{s_n}$ of a physical measure supported on $\Lambda_{s_n}(U)$,
  every weak$^*$ accumulation point $\mu$ of $(\mu^{s_n})_{n\ge1}$ is a
  convex linear combination of the ergodic physical measures of
  $\Lambda_s(U)$.
\end{maintheorem}

The conclusion of the previous theorem means, more precisely, that
\begin{enumerate}[(i)]
\item there are weights $\alpha_i\ge0$ such that $\sum_i\alpha_i=1$
  and $\mu=\sum_{i}\alpha_i\mu_i^s$; and
\item we have
  $ \left| \int\vfi\,d\mu_n -
    \sum_{i}\alpha_i\int\vfi\,d\mu_i^s\right|
  \xrightarrow[n\to\infty]{}0$ for every continuous observable
  $\vfi:U\to\RR$.
\end{enumerate}
In the applications which we present in what follows, the property
stated in item (2) above is provided by the (Pesin's) Entropy
Formula~\cite{Pe77,Man81,LY85}, that is, the potential is a geometric
potential $\psi_s=\log|\det|D\psi_1^{G_s}\mid E^{cu}|$ where $E^{cu}$
is a certain continuous subbundle of the tangent bundle at the points
of the attracting set.


\subsection{Application to singular-hyperbolic attracting sets}
\label{sec:singul-hyperb-attrac}

Here we provide open classes of examples of application of the
previous abstract setting: the singular-hyperbolic attracing sets
(also known as ``Lorenz-like attractors''), encompassing, as
particular cases, the classical Lorenz attractor and the geometric
Lorenz attractor.

\subsubsection{Background on singular-hyperbolicity}
\label{sec:backgr-singul-hyperb}

Let $\Lambda$ be a compact invariant set for $G \in
\X^r(M)$.  We say that $\Lambda$ is {\em partially
  hyperbolic} if the tangent bundle over $\Lambda$ can be
written as a continuous $D\phi_t$-invariant sum
$
T_\Lambda M=E^s\oplus E^{cu},
$
where $d_s=\dim E^s_x\ge1$ and $d_{cu}=\dim E^{cu}_x=2$ for $x\in\Lambda$,
and there exist constants $C>0$, $\lambda\in(0,1)$ such that
for all $x \in \Lambda$, $t\ge0$, we have
\begin{itemize}
\item uniform contraction along $E^s$:
$\|D\phi_t | E^s_x\| \le C \lambda^t;$
\item domination of the splitting:
$\|D\phi_t | E^s_x\| \cdot \|D\phi_{-t} | E^{cu}_{\phi_tx}\| \le C \lambda^t.$
\end{itemize}
We refer to $E^s$ as the stable bundle and to $E^{cu}$ as the
center-unstable bundle.  A {\em partially hyperbolic attracting set}
is a partially hyperbolic set that is also an attracting set.

The center-unstable bundle $E^{cu}$ is \emph{volume expanding} if
there exists $K,\theta>0$ such that
$|\det(D\phi_t| E^{cu}_x)|\geq K e^{\theta t}$ for all $x\in \Lambda$,
$t\geq 0$.

We say that $\sigma\in M$ with $G(\sigma)=0$ is an {\em equilibrium}
or \emph{singularity}. In what follows and we denote by $\sing(G)$ the
family of all such points. We say that a singularity
$\sigma\in\sing(G)$ is \emph{hyperbolic} if all the eigenvalues of
$DG(\sigma)$ have non-zero real part.

A point $p\in M$ is \emph{periodic} for the flow $\phi_t$ generated by
$G$ if $G(p)\neq\vec0$ and there exists $\tau>0$ so that
$\phi_\tau(p)=p$; its orbit
$\cO_G(p)=\phi_{\RR}(p)=\phi_{[0,\tau]}(p)=\{\phi_tp: t\in[0,\tau]\}$
is a \emph{periodic orbit}, an invariant simple closed curve for the
flow.  An invariant set is \emph{nontrivial} if it is neither a
periodic orbit nor an equilibrium.

We say that a compact nontrivial invariant set $\Lambda$ is a
\emph{singular hyperbolic set} if all equilibria in $\Lambda$ are
hyperbolic, and $\Lambda$ is partially hyperbolic with volume
expanding center-unstable bundle.  A singular hyperbolic set which is
also an attracting set is called a {\em singular hyperbolic attracting
  set}. An \emph{attractor} is a transitive attracting set, that is,
an attracting set $\Lambda$ with a point $z\in\Lambda$ so that its
$\omega$-limit
\begin{align*}
  \omega(z)=\left\{y\in M: \exists t_n\nearrow\infty\text{  s.t.
  } \phi_{t_n}z\xrightarrow[n\to\infty]{}y \right\}
\end{align*}
coincides with $\Lambda$.
 
\subsubsection{Singular-hyperbolicity and statistical stability}
\label{sec:singul-hyperb-statis}

We may now state the following.

\begin{maincorollary}
  \label{mcor:exSingHyp}
  Every singular-hyperbolic attracting set for a $C^2$ flow admits a
  neighborhood $\V$ in $\X^2(M)$ where every system is
  statistically stable.
\end{maincorollary}

More precisely, given a flow $G$ of class $C^2$ on a compact manifold
exhibiting a singular-hyperbolic attracting set $\Lambda$, then we can
find a neighborhood $\V$ of $G$ in $\X^2(M)$ and a neighborhood $U$
of $\Lambda$ so that, letting $\cG:\V\to\X^2(M)$ be the restriction of
the identity to $\V$, then $\cG$ satisfies the conditions of
Theorem~\ref{mthm:statstabeqstates}. Indeed: for each $Y\in \V$ we
have that $\Lambda_Y(U)$ is a singular-hyperbolic attracting set and
\begin{enumerate}
\item
  $\Psi(Y,x)=\log|\det D\phi_1^Y\mid E^{cu}_x|,
  x\in\Lambda_Y(U)$ is continuous on $W(U)$ as in
  Theorem~\ref{mthm:statstabeqstates}(3) by robustness and continuity
  of dominated splittings in the $C^2$ neighborbood $\V$ -- see
  e.g. \cite[Appendix B]{BDV2004};
\item there are finitely many ergodic physical measures
  $\mu^Y_i, i=1,\dots,k(Y)$ supported in $\Lambda_Y(U)$ whose basins
  cover a full volume subset of $U$ -- see e.g. \cite{ArSzTr,ArMel18};
\item each physical measure supported in $\Lambda_Y(U)$ is an
  equilibrium state with respect to the potential
  $\psi_Y(x)=\psi(Y,x)$ -- see e.g. \cite{ArSzTr} again; and
  \item $\cG$ is robustly expansive: this was recently obtained
    in~\cite{ArCerq}.
\end{enumerate}
In the particular case of the classical Lorenz attractor~\cite{Lo63},
which was shown to be a robustly transitive singular-hyperbolic
attractor with the features of the geometrical Lorenz
attractor~\cite{Tu99}, we have a unique physical measure which has
strong statistical properties~\cite{APPV,AMV15,ArMel16} on a $C^2$ neighborhood
$\V$ as above. That is, we have (1-4) with $k(Y)\equiv1$. Hence we
reobtain a version of the main result of~\cite{bahsoun_ruziboev}:
\begin{corollary}
  \label{cor:Ruz}
  In a $C^2$ neighborhood $\V$ of a geometric Lorenz attractor with
  trapping region $U\subset\RR^3$, if $Y_n\to Y$ in the $C^2$ topology
  of $\X^2(\RR^3)$, then the
  unique physical measures supported on the attractors satisfy
  $\lim_{n\to\infty}\int\vfi\,d\mu_{Y_n} = \int\vfi\,d\mu_Y$ for all
  continuous observables $\vfi:U\to\RR$.
\end{corollary}

\subsection{Application to the Contracting Lorenz (Rovella) attractor}
\label{sec:rovella}

Here we provide a non-trivial example of application of the abstract
setting of the Main Theorem where the family of dynamics is not open:
perturbation of the Rovella or Contracting Lorenz attractors,
presented by Rovella in \cite{Ro93}.

\subsubsection{Background on the contracting Lorenz attractor}
\label{sec:backgr-contract-lore}

To present this dynamics and its main features, we start with the
geometric contracting Lorenz Flow, which is a modification of the
geometric Lorenz attractor from~\cite{Gu76,GW79,ABS77}, in which the
uniformly expanding direction at the singularity is replaced by a
strict nonuniformly expanding direction. In broad terms, following
\cite{PT09,GNP18}, we start with a linear vector field
$(\dot x, \dot y, \dot z)=(\lambda_1 x,\lambda_2 y, \lambda_3 z)$ in
the cube $[-1,1]^3$ whose real eigenvalues
$\lambda_1,\lambda_2,\lambda_3$ of the singularity at the origin
satisfy
\begin{align*}
  -\lambda_2 > - \lambda_3> \lambda_1 > 0,
  \quad
  r= - \frac{\lambda_2}{\lambda_1},
  \quad
  s=- \frac{\lambda_3}{\lambda_1},
  \qand r>s +3.
\end{align*}
We note that $\lambda_1+\lambda_3<0$ while in the geometric Lorenz
attractor the construction starts with $\lambda_1+\lambda_3>0$; see
e.g. \cite[Chapter 3, Section 3]{AraPac2010}.

Setting
$\Sigma^{-}=[-1/2,0]\times[-1/2,1/2]\times\{1\};
\Sigma^{+}=[0,1/2]\times[-1/2,1/2]\times\{1\}$; and
$\Sigma=\Sigma^{+}\cup \Sigma^{-}$ we have a cross-section for the
linear flow; see the left hand side of Figure~\ref{fig:L3D}. It is
straightforward to calculate the Poincar\'e map from $\Sigma^\pm$ to
the cross-section $x=\pm1$: for the $+$ case we obtain
$(x,y,1)\mapsto (1, yx^r,x^s)$.

\begin{figure}[h]
\begin{center}
  \includegraphics[width=6cm]{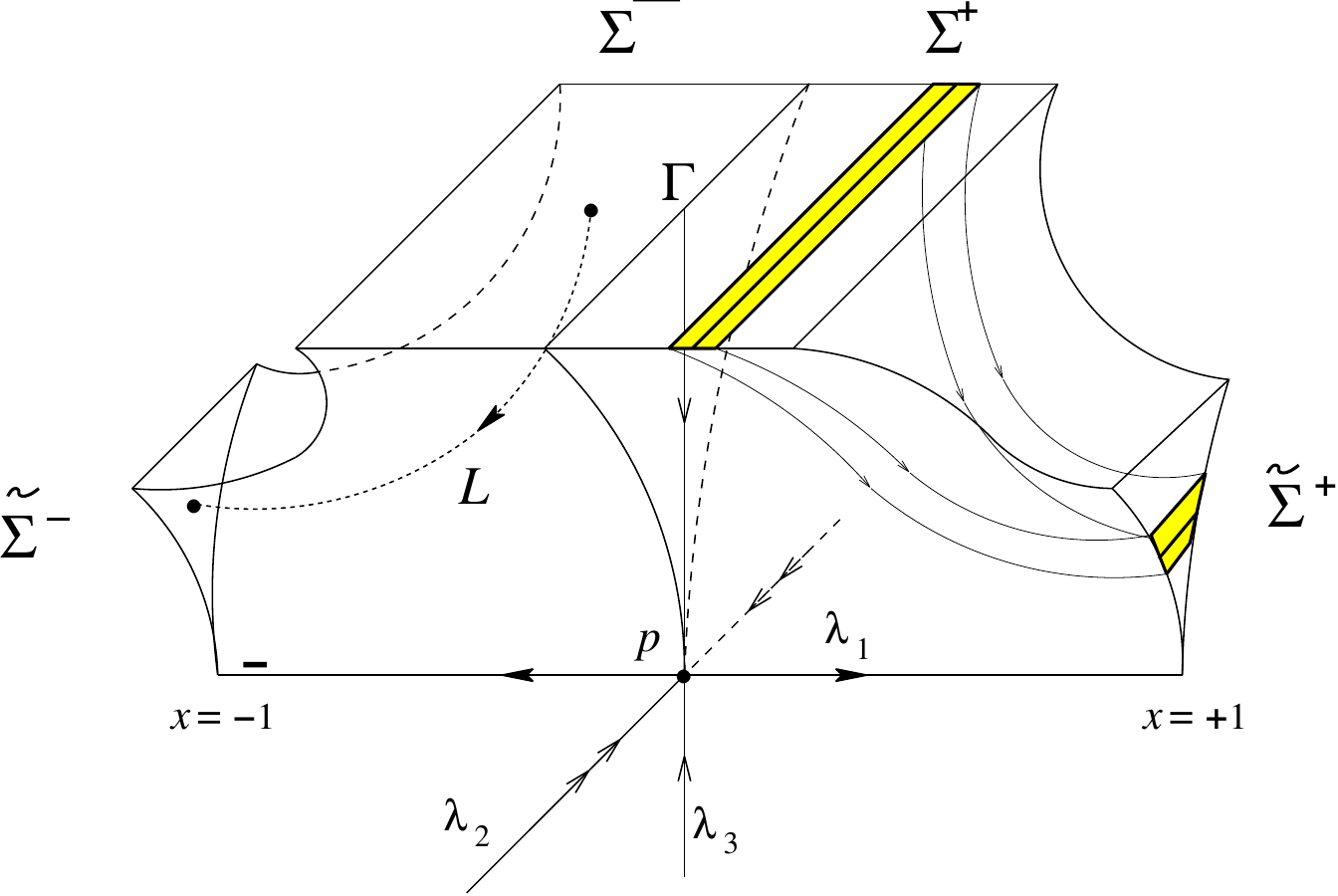}
  \;
  \includegraphics[width=4cm]{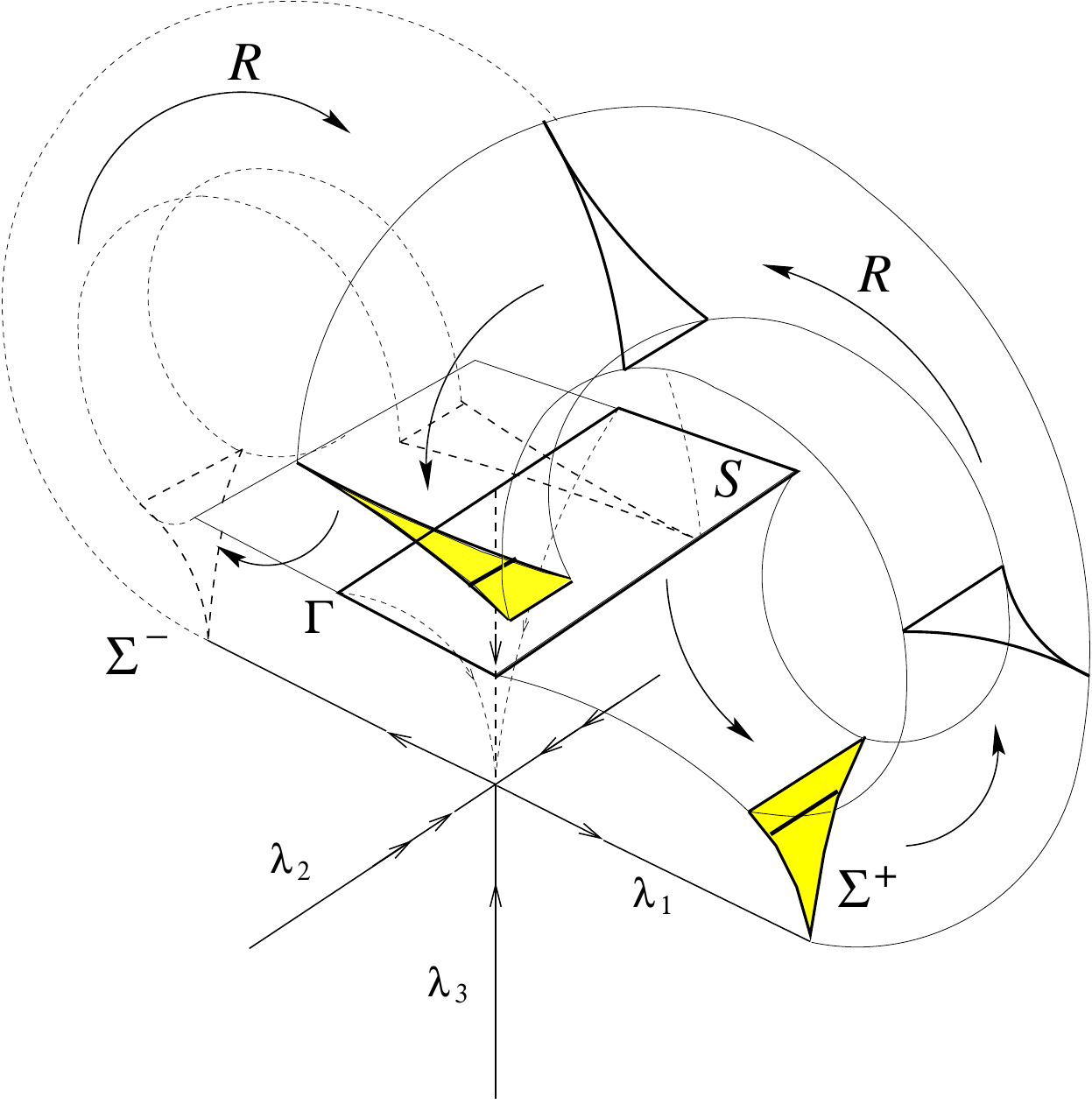}
  \;
  \includegraphics[width=3.5cm]{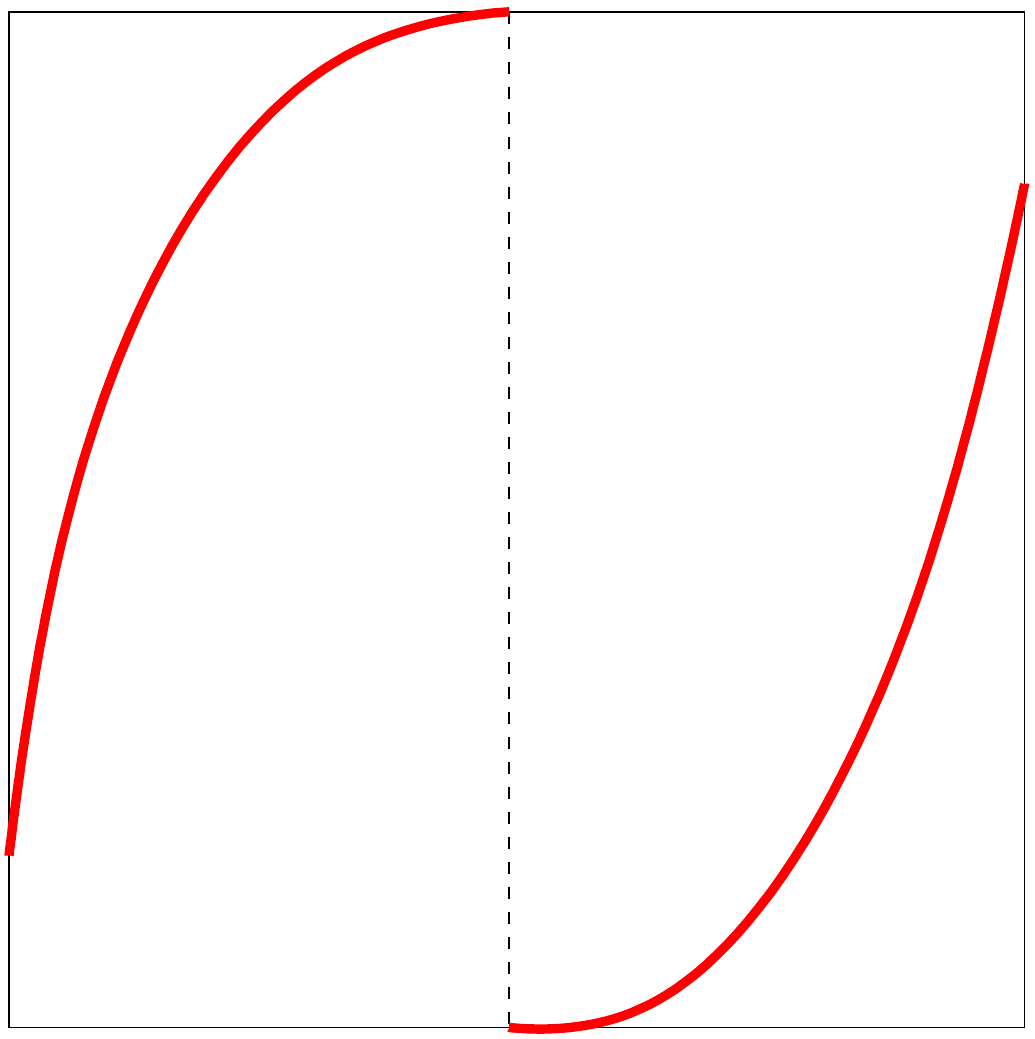}
\end{center}
\caption{\label{fig:L3D}Sketch of behaviour of the linear flow near the origin, on
  the left hand side; and the behaviour of the flow in a neighborhood of
  the attractor, on the center; and the one-dimensional map $T_0$ on
  the right hand side.}
\end{figure}

Outside the cube, we obtain a butterfly shape for the attractor after
rotating the orbits around the origin and returning to $\Sigma$, by a
suitable composition of a rotation, an expansion and a translation;
see the center of Figure~\ref{fig:L3D} and for more details,
see e.g. \cite[Chapter 3, Section 3]{AraPac2010}.
\begin{remark}\label{rmk:smooth-stable}
  As shown in \cite{Ro93} the condition $r>s+3$ ensures the existence
  of a $C^3$ uniformly contracting stable foliation for the
  Poincar\'{e} first return map of all small enough perturbations of
  the contracting geometric Lorenz flow.
\end{remark}
Using this foliation it is possible to obtain
an explicit expression for the Poincar\'e first return map
$R_0(x,y)=(T_0x,H_0(x,y))$ where
\begin{align*}
  T_0(x)=\sgn(x)\cdot(-\rho|x|^s+1/2)
  \qand
  H_0(x,y)=\sgn(x)\cdot(y|x|^r+c)
\end{align*}
for some $c>0$ depending on the choice of the rotations and
translations (assumming some symmetry to simplify the exposition), $r$
and $s$ are as defined above, and $0<\rho \leq (1/2)^{-s}$.

In \cite[Item 4, page 240]{Ro93} it is shown that $T_0$ satisfies (see
the right hand side of
Figure~\ref{fig:L3D})
\begin{enumerate}
\item $T_0$ is piecewise $C^{3}$ with two branches, restricted to each
  it is onto, and $T_0^{\prime }(x)=O(x^{s-1})$ at $x=0$\footnote{We
    write $f(x)=O(g(x))$ at $x=x_0$ if there exists $M,\delta$ such
    that $|f(x)|\leq M |g(x)|$ when $0<|x-x_0|<\delta$.} where
  $s-1>0$;
\item $T_0(0^{+})=1/2$ and $T_0(0^{-})=-1/2$;
\item $T_0'<0$ on $[-1/2,1/2]\setminus\{0\}$;
\item $\max T_0'\mid_{(0,1/2]}=T_0'(1/2)$ and
  $ \max T_0'\mid_{[-1/2,0)}=T_0'(-1/2)$.
  
Moreover, there are values of $\rho\leq (1/2)^{-s}$ so that
\item $\pm1/2$ are preperiodic repelling for $T_0$; and
\item $T_0$ has negative Schwarzian derivative.\footnote{This
    technical condition was strongly used to derive the stated
    results; see \cite[Remarks, p. 240]{Ro93}.}
\end{enumerate}

Rovella established that the flow of the vector field $G_0$ with these
features has an attractor $\Lambda_0$ and studied the dynamics of the
perturbations of this flow.  To state the results more relevant to us,
we present the notion of measure theoretical stability (persistence)
among parametrized families of systems.

We recall that a point $x$ is a \emph{density point} of a subset $S$
of a finite dimensional Riemannian manifold $M$, if
$$
\lim_{r\to 0}\frac{\m(B_r(x)\cap S)}{\m(B_r(x))}=1,
$$
where $B_r(x)$ the ball of radius $r$ centered at $x$.

\begin{definition}
  \label{def:kdensity}
  Given a subset $S$ of a Banach space $X$, we say that $x \in S$ is a
  point of $k$-dimensional full density of $S$ if there exists a
  $C^\infty$ submanifold $N\subset X$ with codimension $k$, containing
  $x$, such that every $k$-dimensional manifold $M$ intersecting $N$
  transversally at $x$ admits $x$ as a full density point of
  $S\cap M$ in $M$.
\end{definition}

We may now state what is mean by a \emph{persistent} attractor.

\begin{definition}
  \label{def:persistence}
  An attractor $\Lambda$ of a vector field $X\in \X^{\infty}$ is
  $k$-dimensionally almost persistent if it has a local basin $U$ such
  that $X$ is a $k$-dimensional full density point of the set of
  vector fields $Y\in \X^{\infty}$, for which
  $\Lambda_Y = \cap_{t > 0}Y^t(U)$ is an attractor.
\end{definition}

In \cite[item (b) at page 235]{Ro93} it is stated (and later proved in
the same work) that the attractor $\Lambda_0$ constructed as above is
$2$-dimensionally almost persistent in the $C^3$ topology.  Recently
this attractor was shown to be a prototype of a class of invariant
sets, similarly to the geometric Lorenz attractor, which is a
prototype of a singular-hyperbolic set.

\begin{definition}
  \label{def:ASH}
  A compact invariant partially hyperbolic set $\Lambda$ of a vector
  field $G$ (in the same setting as
  subsection~\ref{sec:backgr-singul-hyperb}, i.e. $d_{cu}=2$), whose
  singularities are hyperbolic, is \emph{asymptotically sectional
    hyperbolic} if the center-unstable subbundle is eventually
  asymptotically expanding outside the stable manifold of the
  singularities. That is, there exists $c_*>0$ so that
  \begin{align*}
    \limsup_{T\to\infty}\frac1T\log|\det(D\phi_T\mid E^{cu}_x)|\ge c_*,
    \quad\text{for each  } x\in\Lambda\setminus
  \bigcup_{\sigma\in\Lambda\cap\sing(G)}W^s(\sigma).
  \end{align*}
\end{definition}
Here $W^s(\sigma)=\{x\in M: \lim_{t\to+\infty}\phi_tx = \sigma\}$ is
the \emph{stable manifold} of the hyperbolic equilibrium $\sigma$. It
is well-known that $W^s(\sigma)$ is a immersed submanifold of $M$;
see e.g.\cite{PM82}.

The following was recently proved in~\cite{smvivas}.

\begin{theorem}
  \label{thm:rovella-ASH}
  The attractor $\Lambda_0$ is $2$-dimensionally almost persistent
  asymptotically sectional hyperbolic in the $C^3$ topology.
\end{theorem}

Let $R\subset\X^3(\RR^3)$ be the set of vector fields exhibiting a
Rovella attractor provided by Theorem~\ref{thm:rovella-ASH} and
$\cG:R\to\X^3(\RR^3)$ be the restriction of the identity to $R$.

\begin{maintheorem}
  \label{mthm:rovella-ss}
  The family $\cG$ of contracting Lorenz attractors, with trapping
  region $U$, is such that each of its elements admits a unique
  physical measure, whose basin covers $U$ except for zero $\m$-measure
  subset and is statistically stable.
\end{maintheorem}

\subsection{Organization of the text}
\label{sec:organization-text}

The strategy of the proof of Theorem~\ref{mthm:rovella-ss} is to show
that $\cG$ above satisfies all the conditions of
Theorem~\ref{mthm:statstabeqstates} with uniqueness of physical
measures for each attractor. This is presented in
Section~\ref{sec:contract-lorenz-fami}. In
Section~\ref{sec:proof-statist-stabil} we provide a proof of
Theorem~\ref{mthm:statstabeqstates}.

 \subsection*{Acknowledgments}
 \label{sec:acknowledgments}

 I thank the Mathematics Department at UFBA; CAPES-Brazil and
 CNPq-Brazil for the basic support of research activities; and also
 the anonymous referee for many suggestions that helped to improve the
 text.


\section{The contracting Lorenz family of attractors}
\label{sec:contract-lorenz-fami}

Here we prove Theorem~\ref{mthm:rovella-ss} by showing that the family
perturbations of the attractor $\Lambda_0$ introduced by Rovella,
known also as contracting Lorenz attractors, satisfies the conditions
for statistical stability in the weak$^*$ topology stated in
Theorem~\ref{mthm:statstabeqstates}, with unique physical measures for
each element of the family $\cG$ in the statement of
Theorem~\ref{mthm:rovella-ss}.

\subsection{Existence and uniqueness of physical measure}
\label{sec:existence-uniquen-ph}

We start by observing that the partial hyperbolicity of the family of
contracting Lorenz flows given by \ref{thm:rovella-ASH} implies that
there exists an $D\phi_t$-invariant and uniformly contracting
extension of the subbundle $E^s$ to $U$ (which we denote by the same
symbol) together with $\epsilon_0>0$ such that, for all points
$x\in U$ and $0<\epsilon<\epsilon_0$, there exists a $C^3$ embedded disk
\begin{align*}
  W^{ss}_\epsilon(x)=\{y\in B(x,\epsilon):d(\phi_ty,\phi_tx)\xrightarrow[t\to+\infty]{}0\}
\end{align*}
which satisfies $T_xW^s_{\epsilon}(x)=E^s_x$ and is
$\phi_t$-invariant, that is
$\phi_tW^s_\epsilon(x)\subset W^s_\epsilon(\phi_tx)$ for all $t>0$;
see \cite{ArMel17}. In what follows, this disk is the \emph{local
  (strong-)stable manifold of size $\epsilon$ of $x$} and, when we do
not want to specify its size, we write $W^s_{loc}(x)$ understanding
that the size is to be taken uniform in $U$. It follows from the
theory of uniform hyperbolicity that $\epsilon_0>0$ above may be taken
uniformly on $U$ and on the vector field $G$ on a neighborhood $\V$ of
$G_0$; see \cite{PM82}.

Using the results from \cite{Ro93}, we have that, in a $C^3$
neighborhood $\V$ of the vector field $G_0$ described in
Subsection~\ref{sec:backgr-contract-lore}, the Poincar\'e first return
map $R_G$ to the cross-section $\Sigma$ for each $G\in\V$ can be
written as a skew-product $R_G(x,y)=(T_Gx,H_G(x,y))$ after a suitable
$C^3$ change of coordinates; this is a consequence of
Remark~\ref{rmk:smooth-stable}.

As proved in \cite{Ro93}, used in \cite{mtz001} and generalized
recently in \cite{AlvesSoufi12}, there exists a one-parameter family
$G_a, a\in[0,1]$ of vector fields $C^3$ close to $G_0$ admiting a subset
$E\subset(0,a_0)$ of parameters (``Rovella parameters'') so that $0$
is a density point of $E$. Moreover, the one-dimensional map $T_a$
corresponding the quotient $T_{G_a}$ of the Poincar\'e return map to
$\Sigma$ over the stable foliation, satisfies the following.

\begin{theorem}{\cite{AlvesSoufi12}}
  \label{thm:unique-acim}
  For each $a\in E$, the map $T_a$ of the interval $[-1/2,1/2]$ is a
  transitive non-uniformly expanding map with slow recurrence to the
  critical set; and has a unique absolutely continuous ergodic
  invariant probability measure $\nu_a$, whose basin $B(\nu_a)$ equals
  $[-1/2,1/2]$ except for a subset of zero Lebesgue measure.
\end{theorem}

As explained in \cite[Section 7]{mtz001} and also
e.g. in~\cite[Section 6]{APPV}, the existence of an ergodic physical
measure for the quotient map $T_a$ of a Poincar\'e return map $R_a$
over a uniformly contracting regular foliation, induces an ergodic
physical invariant probability measure for the flow through a standard
procedure.  In addition, if we start with a physical measure with full
ergodic basin for $T_0$, then the induced measure also has full
ergodic basin over the orbits of the flow starting on the
cross-section, which we may assume without loss of generality to
include $U$.

Hence, the flow of $G_a$ on the trapping region $U$ admits a physical
invariant probability measure $\mu_a$ supported on
$\Lambda_a=\Lambda_{G_a}(U)$ with full ergodic basin on $U$. Thus this
measure is the unique physical measure on $U$. We have obtained item
(1) of the statement of Theorem~\ref{mthm:statstabeqstates} with a
unique measure for each element of the family $\cG$.







\subsection{The physical measure is a $SRB$ measure}
\label{sec:physical-measure-cu}

Let $R_a$ be the Poincar\'e first return map to $\Sigma$.
As presented in \cite[Section 8]{APPV} or \cite[Chapter 7, Sections
9-11]{AraPac2010}, if we assume that
\begin{itemize}
\item the Poincar\'e return map $R_a(x,y)=(T_ax,H_a(x,y))$ to the
  cross-section $\Sigma$ satisfies
  \begin{itemize}
  \item $H_a(x,\cdot)$ is a uniform contraction;
  \item $T_a$ is a one-dimensional non-uniformly expanding map with
    slow recurrence to the critical set (this is the discontinuous
    point $\{0\}$), as defined in \cite[Section 5]{ABV00} and provided
    by~\cite{AlvesSoufi12};
  \end{itemize}
\end{itemize}
then \emph{every absolutely continuous ergodic $T_0$-invariant
  probability measure $\nu_a$ induces a measure $\mu_a$ which is an
  ergodic hyperbolic $SRB$-measure}. That is, $\mu_a$ admits an
absolutely continuous disintegration along unstable manifolds.

\begin{remark}\label{rmk:hypmeas}
  Observe that since the flow direction on partially hyperbolic sets
  is contained in the central-unstable direction (see e.g.~\cite[Lemma
  5.1]{ArArbSal}), then Oseledets Theorem ensures that
  \begin{align*}
    \int\log|\det(D\phi_1^{G_a}\mid
    E^{cu})|\,d\mu_a=\int\lambda^+(x)\,d\mu_a(x)\ge c_*>0,
  \end{align*}
  where
  $\lambda^+(x)=\lim_{T\to\infty}\frac1T\log|\det(D\phi_T\mid
  E^{cu}_x)|$ is the largest Lyapunov exponent along the
  two-dimensional bundle $E^{cu}$ for $\mu_a$-a.e. $x$. This is
  strictly positive by asymptotical sectional-expansion, for
  otherwise a $\mu_a$ generic point would belong to the stable
  manifold of a singularity $\sigma\in\Lambda_a$, and thus 
  $\mu_a=\delta_\sigma$. But this would contradict the $SRB$ property
  obtained above.
\end{remark}

According the characterization of $SRB$ measures obtained by
Ledrappier and Young~\cite{LY85} we have
$ h_{\mu_a}(\phi_1^{G_a}) =\int\lambda^+(x)\,d\mu_a(x)$ and so after
Remark~\ref{rmk:hypmeas} we see that $\mu_a$ satisfies the Entropy
Formula
\begin{align}\label{eq:entropyform}
  h_{\mu_a}(\phi_1^{G_a})
  =\int\log|\det(D\phi_1^{G_a}\mid
  E^{cu})|\,d\mu_a>0.
\end{align}
Reciprocally, an invariant probability measure $\mu$ satisfying the
Entropy Formula~\eqref{eq:entropyform} for the $C^2$ partially
hyperbolic flow $G_a$ is a $SRB$ measure (by the result
from~\cite{LY85}) and since $E^{cu}$ is two-dimensional, then $\mu$ is
a hyperbolic measure: the Lyapunov exponents along $E^s$ are strictly
negative, there exists a positive Lyapunov exponent along the $E^{cu}$
direction together with the zero exponent along the flow
direction. Consequently, being a $SRB$ and hyperbolic measure, it is a
physical measure; see e.g. \cite{PS82,Young2002}.

Hence, using the the potential
$\psi_a=-\log|\det(D\phi_1^{G_a}\mid E^{cu})|$ we obtain item (2) of
the statement of Theorem~\ref{mthm:statstabeqstates}.

The continuity of dominated splittings \cite[Appendix B]{BDV2004} with
respect to the base point but also with respect to the dynamics,
together with the $C^3$ smoothness of the vector fields involved,
ensures that item (3) also holds in this setting.

\subsection{Robust expansiveness of contracting Lorenz flows}
\label{sec:robust-expansiveness}

Here we deduce robust expansiveness. We first use the following result
from \cite[Section 4]{mtz001}. We write
$c_a^\pm=T_a(0^\pm)=\lim_{t\to0^\pm}f(t)$; see the right hand side of
Figure~\ref{fig:L3D}. We note that $c_a^-<0<c_a^+$ and
$c_a^\pm\to\pm1/2$ when $a\to0$.

\begin{lemma}{\cite[Lemma 4.1]{mtz001}}
  \label{le:leo}
  There exists a $C^3$ neighborhood $\V$ of $G_0$ so that
  if $G_a\in\V$, then the map $T_a$ is locally eventually onto, that
  is, for any interval $J\subset [-1/2,1/2]\setminus\{0\}$ there
  exists $n=n(J)>0$ so that $f^n(J)\subset[c_a^-,c_a^+]$.

  Consequently, \emph{there does not exist a pair of points
  $x_0<y_0$ with the same sign in $[-1/2,1/2]\setminus\{0\}$ so that
  $T_a^n[x_0,y_0]$ does not contain the origin for all $n\ge1$.}
\end{lemma}

We use this result to obtain robust expansiveness for the family $\cG$
restricted to the neighborhood $\V$.

Let $2\delta_0>0$ be the distance between the cross-sections
$\widetilde{\Sigma^+}$ and $\Sigma^+$, or between
$\widetilde{\Sigma^-}$ and $\Sigma^-$ (they are symmetrical); see the
left hand side of Figure~\ref{fig:L3D}. Let also $x,y\in U$ and
$h:\RR\to\RR$ be a surjective increasing continuous function such that
$d(x(t),y(t))\leq\delta$ for some $\delta\in(0,\delta_0)$ and for all
$t\in\RR$, where $x(t)=\phi_tx$ and $y(t)=\phi_{h(t)}y$ will be the
trajectories to consider in what follows (we removed $G_a$ from the
notation of the flow to lighten the text).

We consider also the pairs of consecutive hitting times
$x_n,y_n, n\ge1$ of these trajectories on $\Sigma$ and their
projections $\pi x_n,\pi y_n$ on the quotient $[-1/2,1/2]$ of $\Sigma$
over the stable leaves.

We note that if $\pi x_n\cdot \pi y_n<0$, i.e. returns to $\Sigma$ lie
on different sides with respect to the stable manifold of the
singularity at the origin, then the trajectories of $x_n$ and $y_n$
will eventually separate by a distance larger than $\delta$ during
their crossing of the linearized region near the singularity; see
again the the left hand side of Figure~\ref{fig:L3D}. This would
contradict the assumption on $x,y$ and $h$.

However, if we assume that $\pi x_1<\pi y_1$ and
$\pi x_1\cdot \pi y_1>0$, then, because $T_a$ has monotonous smooth
branches on $[-1/2,0)$ and $(0,1/2]$, we get
$[\pi x_{j+1},\pi y_{j+1}]=T_a[\pi x_{j},\pi y_{j}]$ as long as
$\pi x_j\cdot \pi y_j>0$ for $j=1,\dots,l$. Hence, from
Lemma~\ref{le:leo}, the trajectories will not be in this situation for
all $j\ge1$: there exists $l>1$ so that
$\pi x_{j+1}\cdot \pi y_{j+1}<0$. Hence $x,y$ cannot satisfy
$\pi x_1\cdot \pi y_1<0$ nor $\pi x_1\neq\pi y_1$.

We conclude that $\pi x_1=\pi y_1$ and both trajectories share a
stable leaf of the Poincar\'e return map $R_a$. This means that there
exists $t_1>0$ and $h(s_1)>0$ so that
$x_1=\phi_{t_1}x, y_1=\phi_{h(s_1)}y$ and
$d(x_1,\phi_{h(t_1)}y)<\delta$, and also that $y_1$ is in the same
contracting leaf of $R_0$ as $x_1$. Hence, since there exists $d>0$ so
that $\|G_a\|>d$ in a $\delta$-neighborhood of $\Sigma$ and the
curvature of the trajectories within this neighborhood is uniformly
bounded, \emph{for all $a$ such that $G_a\in\V$}, we have
\begin{enumerate}
\item there exists a constant $K>0$\footnote{This depends only on the
    \emph{neighborhood $\V$} through $d$.} so that
  $|h(s_1)-h(t_1)|<K\delta$; and
\item there exists $\epsilon_1>0$\footnote{This follows from the
    invariance of the stable manifolds of all points in $U$ together
    with the closeness of $y_1$ and $x_1$, together with the value of
    $d$ in the neighborbood $\V$.} and $t\in(-\epsilon_1,\epsilon_1)$ such that
  $\phi_ty_1\in W^{ss}_{loc}(x_1)$.
\end{enumerate}

Therefore
$\phi_{h(t_1)}y\in \phi_{ [h(s_1)-K\delta,h(s_1)+K\delta]
}(y)=A(y,\delta)$.

Let $\epsilon>0$ be given, set
$A(x,\epsilon)=\phi_{[t_1-\epsilon,t_1+\epsilon]}x$ and consider the
set of points of the trajectory of $x$ whose stable manifolds contain
points of $A(y,\delta)$
\begin{align*}
  A(x,y,\delta)=\{\phi_sx: W^{ss}_{loc}(\phi_sx)\cap A(y,\delta)\neq\emptyset\}.
\end{align*}
From item (2) above, we have that $A(x,y,\delta)$ is a neighborhood of
$\phi_{t_1}x$. This neighborhood can be made smaller by reducing
$\delta>0$ so that $A(x,y,\delta)\subset A(x,\epsilon)$. This means
that
\begin{align*}
\phi_{h(t_1)}(y)\in W^{ss}_{loc}(\phi_sx)\quad\text{for some}\quad
s\in[t_1-\epsilon,t_1+\epsilon].
\end{align*}
This is enough to conclude robust expansiveness. Indeed, following
\cite[Section 3.1]{APPV} we state first an auxiliary result.

\begin{lemma}{\cite[Lemma 3.2]{APPV}}\label{le:near}
There exist $c>0$ and $\rho>0$, depending only on the flow,
such that if $z_1, z_2, z_3$ are points in $U$ satisfying
$z_3\in \phi_{[-\rho,\rho]}(z_2)$ and $z_2\in W_\rho^{ss}(z_1)$, then
\[
d(z_1,z_3) \ge c \cdot \max\{d(z_1,z_2),d(z_2,z_3)\}.
\]
\end{lemma}
We may assume without loss of generality that
$100\delta<cd\rho$. Arguing by contradiction, if
$\phi_{h(t_1)}(y)\neq\phi_{s}(x)$, then there exists a largest
$\theta>0$ satisfying
\begin{align*}
  \phi_{h(t_1)-t}(y)\in W^{ss}_\rho(\phi_{s-t}x)
  \qand
  \phi_{h(s-t)}(y)\in\phi_{[-\rho,\rho]}(\phi_{h(t_1)-t}(y))
\end{align*}
for all $0\le t\le\theta$. Hence for $t=\theta$ we must have
\begin{itemize}
\item either $d(\phi_{h(t_1)-t}(y),\phi_{s-t}(x))=\rho$;
\item or $d(\phi_{h(t_1)-t}(y), \phi_{h(s-t))}(y))\ge \frac12 d\rho$.
\end{itemize}
From Lemma~\ref{le:near} we deduce that
$d(\phi_{s-t}x,\phi_{h(s-t)}y)\ge cd\rho/2 >\delta$ contradicting the
assumption on $x,y$ and $h$.

We have prove expansiveness for any pair $\epsilon>0$ and
$\delta<\min\{cd\rho/100,\delta_0\}$, where all the constants involved
in the estimates are uniform in a neighborhood $\V$ of $G_0$, as
needed for robust expansiveness.

Altogether, the results in this section complete the proof of
Theorem~\ref{mthm:rovella-ss}.

\section{Proof of Statistical Stability}
\label{sec:proof-statist-stabil}

Here we prove the result on statistical stability for families of
flows in the conditions stated in the Main Theorem.  In the following
statements $X,M$ denote compact metric spaces.


\begin{theorem}[Continuity of equilibrium states]
\label{thm:conteqstate}
Let $f:X\times M\to M$ and $\psi:X\times M\to\RR$ be continuous maps,
which define a family of continuous maps
$f_t:M\to M, y\in Y\mapsto f_t(y)=f(t,y), t\in X$ and continuous
potentials $(\psi_t)_{t\in X}$ satisfying the following conditions.
\begin{enumerate}
\item $f_t$ admits some equilibrium state for $\psi_t$, i.e.
  there exists $\mu_t\in\cP_{f_t}(M)$ such that
  $P_{f_t}(\psi_t)=h_{\mu_t}(f_t)+\int \psi_t \, d\mu_t$ for
  all $t\in X$.
\item For each weak$^*$ accumulation point $\mu_0$ of $\mu_t$ when
  $t\to *\in X$, let $t_k\to *$ when $k\to\infty$ be such that
  $\mu_k=\mu_{t_k}\to\mu_0$. We write
  $f_k=f_{t_k}, \, \psi_k=\psi_{t_k}$ and assume also that
  \begin{enumerate}
  \item there exists a finite Borel partition
    $\xi$ of $M$ such that $h_{\mu_k}(f_k)=h_{\mu_k}(f_k,\xi)$ for all
    $k\ge1$; 
    and $\mu_0(\partial \xi)=0$.
  \item $P_{f_k}(\psi_k)\to P_{f_*}(\psi_*)$ when $k\to\infty$.
  \end{enumerate}
\end{enumerate}
Then every weak$^*$ accumulation point $\mu$ of $(\mu_k)_{k\ge1}$ when
$k\to\infty$ is a equilibrium state for $f_*$ and the potential
$\psi_*$.
\end{theorem}

Theorem~\ref{thm:conteqstate} is already known in several versions for
applications both to statistical and stochastic stability; see
e.g.~\cite[Theorems 10-12]{araujo2006} and also~\cite{CoYo2004} and
\cite{ArTah,ArTah2}. For completeness we provide its short proof.

\begin{proof}
  For each fixed $N,k>1$ we have by assumption
  \begin{align*}
    P_{f_k}(\psi_k)
    =
    h_{\mu_k}(f_k,\xi)+\int \psi_k \, d\mu_k
    \le
    \frac1N H_{\mu_k}(\xi_k^N)
    +
    \int \psi_k\, d\mu_k
  \end{align*}
  where $\xi_k^N=\bigvee_{i=0}^{N-1}(f_k^i)^{-1}\xi$.  Letting
  $k\to\infty$ we obtain by assumption (and compactness)
\begin{align*}
  P_{f_*}(\psi_*)
  \le
  \frac1N\limsup_{k\to\infty}H_{\mu_k}(\xi_k^N)
  + \int\psi_*\,d\mu.
\end{align*}
Finally since $\mu(\partial\xi_*^N)=0$ and
$\mu_k(\xi_k^N(x))\to\mu(\xi_*^N(x))$ for $\mu$-a.e. $x$, we obtain
\begin{align*}
  \limsup_{k\to\infty}H_{\mu_k}(\xi_k^N)=H_{\mu}(\xi_*^N)
\end{align*}
and because $N>1$ is arbitrary, we conclude
\begin{align*}
  P_{f_*}(\psi_*)
  \le
  h_{\mu}(f_*,\xi)  + \int\psi_*\,d\mu
  =
  h_{\mu}(f_*)+\int\psi_*\,d\mu,
\end{align*}
which shows that $\mu$ is an equilibrium state for $\psi_*$.
\end{proof}

Now we need to check that the assumptions of
Theorem~\ref{mthm:statstabeqstates} imply the conditions of
Theorem~\ref{thm:conteqstate}.

\subsection{Entropy expansiveness}
\label{sec:local-entropy-entrop}

A way to quantify how the flow of $G$ moves trajectories away from one
another is to use \emph{dynamical balls}. For each $x\in M$ and
$\epsilon>0$ we set for each given $t>0$
\begin{align*}
  B_t(x,\epsilon)
  =
  \bigcap_{|u|<t}\phi_{-u}B(\phi_ux,\epsilon)
  =
  \{y\in M: d(\phi_ux,\phi_uy)<\epsilon, -t<u<t\}.
\end{align*}
We denote $f=\phi_1$ the time-$1$ map of the flow of $G$.
Given $E,F\subset M$ we say that $F (n,\delta)$-spans $E$ if
\begin{align*}
  E\subset\cup_{y\in F}B_n(y,\delta)
\end{align*}
and we set $r_n(E,\delta)$ as the largest number of elements of a
$(n,\epsilon)$-spanning set of $E$. We can now define the entropy of
$f$ over a compact subset $K$ as
\begin{align*}
  h(f,K)=\lim_{\delta\to0}\limsup_{n\to\infty}\frac1n\log r_n(K,\delta).
\end{align*}
Following Bowen \cite{Bowen72} we set
$h_{loc}(f,\delta)=\sup_{x\in M}h(f,B^+(x,\delta))$ where
$B^+(x,\delta)=\cap_{n\ge1}B_n(x,\delta)$. We say that the flow of $G$
is \emph{entropy expansive} if $h_{loc}(f,\delta)=0$ for some
$\delta>0$ and this value of $\delta$ is an $h$-expansiveness
constant.

\begin{theorem}
  \label{thm:Bowen}
  Let $M$ be a compact metric space of finite dimension and $\xi$ a
  Borel partition of $M$ with $\diam(\xi)<\epsilon$. Then, for each
  $f$-invariant probability measure $\mu$ we have
  $h_{\mu}(f)\le h_\mu(f,\xi)+h_{loc}(f,\epsilon)$. In particular,
  $h_\mu(f)=h_\mu(f,\xi)$ if $\epsilon$ is an $h$-expansiveness
  constant for $f$.
\end{theorem}

\begin{proof}
  See \cite[Theorem 3.5]{Bowen72}.
\end{proof}

\subsection{Statistical stability}
\label{sec:statist-stabil}

We are now ready for the proof of the Main Theorem.

\begin{proof}[Proof of Theorem~\ref{mthm:statstabeqstates}]
  Let $\cG:N\to\X^s(M)$ be a family of vector fields admitting a
  trapping region $U$ whose attracting set satisfies the conditions on
  the statement of Theorem~\ref{mthm:statstabeqstates}.

  The continuity assumption of item (3) ensures that we may
  continuously extend $\Psi:W(U)\to\RR$ to $\psi:N\times M\to\RR$
  which clearly satisfies items (1) and (2b) of the statement of
  Theorem~\ref{thm:conteqstate} with $X=N$.

  The robustly expansiveness assumption has the following
  straighforward consequence. For a robustly expansive attracting set
  $\Lambda_G(U)$ on the family $\cG:N\to\X^r(M)$ we can find a pair
  $\epsilon,\delta>0$ so that for each $s\in N$ and
  $x\in \Lambda_s(U)$, there exists $t_0\in\RR$ satisfying
  $B^+(x,\delta)=\cap_{T>1}B_T(x,\delta)\subset
  \phi^{G_s}_{[t_0-\epsilon,t_0+\epsilon]}(x)$.

  In particular, this ensures that $\delta$ is an expansiveness
  constant for each vector field $G_s$ on the invariant compact set
  $\Lambda_s(U)$, $s\in N$; see e.g. \cite[Example 1.6]{Bowen72}.

\begin{proposition}\label{pr:robust-h-exp}
  A robustly expansive attracting set $\Lambda_G(U)$ on a family
  $\cG:N\to\X^r(M)$ admits $\delta>0$ which is a constant of
  $h$-expansiveness for each flow in the family.
\end{proposition}

Hence, item (4) of the statement of
Theorem~\ref{mthm:statstabeqstates} implies assumption (2b) of
Theorem~\ref{thm:conteqstate}, by using
Proposition~\ref{pr:robust-h-exp} together with
Theorem~\ref{thm:Bowen}.

Let then $s_n\in N$ be a sequence converging to $s\in N$ and $\mu_n$ a
physical measure supported in $\Lambda_{s_n}(U)$. Let $\mu$ be a
weak$^*$ accumulation point of $\mu_n$ when $n\to\infty$. To simplify
the notation we still write $\mu_n\to\mu$ (relabeling the indexes if
necessary). According to item (2) of
Theorem~\ref{mthm:statstabeqstates}, each $\mu_n$ is an equilibrium
state for $\psi_{s_n}$ with $\cP_{f_{s_n}}=0$, where
$f_{s_n}=\phi_1^{G_{s_n}}$.  From Theorem~\ref{thm:conteqstate} we
have that $\mu$ is an equilibrium state with respect to $\psi_s$.

From item (2) of Theorem~\ref{mthm:statstabeqstates} again, we have
that $\mu$ is a physical measure. Hence, by item (1) of
Theorem~\ref{mthm:statstabeqstates}, we have a Lebesgue modulo zero
decomposition
\begin{align*}
  B(\mu)\cap U=U\cap\big(\sum_{i} B(\mu)\cap B(\mu_i)\big).
\end{align*}
By definition of physical measure, for each continuous observable
$\vfi:U\to\RR$
\begin{align*}
  \int\vfi\,d\mu
  &=
    \frac1{\m(U)}\int_U \int\vfi\,
    d\left(\lim_{T\to+\infty}\frac1T\int_{0}^{T}\delta_{\phi^{G_s}_t x}\right)
    \,d\m(x)
  \\
  &=
    \sum_{i} \frac{\m(B(\mu)\cap B(\mu_i)\cap
    U)}{\m(U)}\int \vfi\,d\mu_i^s,
\end{align*}
where the limit above is in the weak$^*$ topology of the probability
measures of the manifold.  Thus we conclude that
$\mu=\sum_{i} \frac{\m(B(\mu)\cap B(\mu_i)\cap U)}{\m(U)}\mu_i$
and $\mu$ is a convex linear combination of the ergodic physical
measures supported in $\Lambda_s(U)$ provided by item (1).

This completes the proof of Theorem~\ref{mthm:statstabeqstates}.
\end{proof}

\begin{remark}\label{rmk:infiniteSRB}
  The statement of Theorem~\ref{mthm:statstabeqstates} can be somewhat
  generalized by extending item (1) to admit a countable family of
  ergodic physical probability measures; and extending item (4) to
  require robust $h$-expansiveness of the family of dynamics.
\end{remark}


\bibliographystyle{abbrv}

\bibliography{../../Trabalho/bibliobase/bibliography}

\end{document}